\documentclass[11pt]{amsart}
\scrollmode
\usepackage{amsmath, amsthm, amssymb,amsfonts,mathtools}
\usepackage{verbatim,geometry,graphicx,multirow,array}
\usepackage{algpseudocode}
\usepackage{enumerate}
\usepackage{xy,psfrag}

\input epsf

\newcommand{\Real}{\operatorname{Re}}
\newcommand{\Img}{\operatorname{Im}}

\newcommand{\R}{{\mathbb R}}
\newcommand{\C}{{\mathbb C}}
\newcommand{\Rnxn}{{\R^{n\times n}}}
\newcommand{\Cnxn}{{\C^{n\times n}}}
\newcommand{\bmat}[1]{ \begin{bmatrix}#1\end{bmatrix}}
\newcommand{\smat}[1]{ \left[\begin{smallmatrix} #1 \end{smallmatrix}\right]}
\newcommand{\cont}{{\mathcal C}}
\newcommand{\diag}{\operatorname{diag}}
\newcommand{\var}{\operatorname{var}}

\newcommand{\abs}[1]{\left| #1 \right|}

\renewcommand{\ss}{\scriptstyle}
\def\sddots{\mathinner{\raise3pt\vbox{\hbox{$\ss .$}}
		\raise1.5pt\hbox{$\ss .$}\hbox{$\ss .$}}}

\theoremstyle{plain}
\newtheorem{thm}{Theorem}[section]
\newtheorem{lem}[thm]{Lemma}

\newtheorem{claim}[thm]{Claim}
\theoremstyle{definition}
\newtheorem{rem}[thm]{Remark}
\newtheorem{rems}[thm]{Remarks}

\newcounter{algo}[section]
\renewcommand{\thealgo}{\thesection.\arabic{algo}}

\newcommand{\algo}[3]{\refstepcounter{algo}
\begin{center}\begin{figure*}[h!]
\framebox[\textwidth]{
\parbox{0.95\textwidth} {\vspace{\topsep}
{\bf Algorithm \thealgo : #2}\label{#1}\\
\vspace*{-\topsep} \mbox{ }\\
{#3} \vspace{\topsep} }}
\end{figure*}\end{center}}

\begin{document}

\title{Takagi factorization of matrices \\ depending on parameters and \\
locating degeneracies of singular values}

\author[Dieci]{Luca Dieci}
\address{School of Mathematics, Georgia Institute of Technology,
Atlanta, GA 30332 U.S.A.}
\email{dieci@math.gatech.edu}
\author[Papini]{Alessandra Papini}
\address{Dept. of Industrial Engineering, University of Florence, viale G. Morgagni 40-44, 50134 Florence, Italy}
\email{alessandra.papini@unifi.it}
\author[Pugliese]{Alessandro Pugliese}
\address{Dept. of Mathematics, University of Bari ``Aldo Moro'', Via Orabona 4, 70125 Bari, Italy}
\email{alessandro.pugliese@uniba.it}


\begin{abstract}
In this work we consider the Takagi factorization of a matrix valued function
depending on parameters.  We give smoothness and genericity results and
pay particular attention to the concerns caused by having
either a singular value equal to $0$ or multiple singular values.  For
these phenomena, we give theoretical results showing that their co-dimension
is $2$, and we further develop and test numerical methods to locate in parameter
space values where these occurrences take place.  Numerical study of the density
of these occurrences is performed.
\end{abstract}


\maketitle

\pagestyle{myheadings}
\thispagestyle{plain}
\markboth{L.~Dieci, A.~Papini, A.~Pugliese.}{Takagi factorization}

\section{Introduction. Takagi factorization of a matrix.}\label{Intro}
The Takagi factorization is concerned with complex symmetric matrices, that is one has 
$A\in \Cnxn$, $A^T=A$, and seeks the factorization
$A=USU^T$, where $U\in \Cnxn$ is
unitary and $S$ is diagonal with nonnegative and ordered
entries.  That is:
$$U:\,\ U^*U=\bar U^T U=I\ ,\quad 
S=\diag(\sigma_1,\dots, \sigma_n), \,\ \sigma_1\ge \dots \ge \sigma_n\ge 0\ .$$
The values $\sigma_1,\dots, \sigma_n$ are the singular values of $A$ and $S$ is indeed the standard matrix $\Sigma$ in the SVD of $A$ (see below).

The Takagi factorization has been used in the study of inhomogeneous plane waves in 
\cite{Hayes:planewaves}, it has shown to be of
relevance in the quantum mechanical formulation of decay phenomena in 
\cite{Reid:complexsymmetric}, and it has also recently been used in the Physics
literature in connection to the Bloch-Messiah reduction, to which is effectively
equivalent; see \cite{BlochMessiah1} and \cite{BlochMessiah2}, and we note that the latter
work is in particular interested in studying the structural changes brought upon by
degeneracies (i.e., multiple singular values), which is precisely the topic we address in the
present work.

The existence of a Takagi factorization of a matrix is a well known fact; e.g., see
\cite{HJ1}, or \cite{BG}, where it is called {\emph{symmetric SVD}}, or \cite{BlochMessiah1}.
For completeness, we next state a result about its existence, further clarifying the
degree of uniqueness of the factorization.

\bigskip
\begin{lem}[Takagi Factorization]\label{UniquenessTakagi}$\,$\\
	{\bf Existence}.
	Let $A\in \Cnxn$ be symmetric: $A=A^T$.  Then, $A$ admits a Takagi factorization
	$A=USU^T$, where $U$ is unitary and 
	$S$ is the diagonal matrix of singular values of $A$, which we will
	take to be ordered: $S=\diag(\sigma_1,\dots, \sigma_n), \,\ \sigma_1\ge \dots \ge \sigma_n\ge 0$.\\
	{\bf Uniqueness}.
	Next, let $A=Q S V^*$ be an SVD of $A$.
	Assume that there are $p$ distinct singular values $\sigma_j$, of multiplicity $n_j$, $j=1,\dots, p$, 
	so that $S=\diag(\sigma_1 I_{n_1}, \dots, \sigma_p I_{n_p})$, $n_1+\dots + n_p=n$.  
	Then, all possible Takagi factorizations are obtained by choosing $U=QK$, where $K$ is
	a block-diagonal unitary matrix, $K=\smat{K_{11} & & \\ & \sddots & \\ & & K_{pp}}$, and each 
	$K_{jj}\in \C^{n_j \times n_j}$	is unitary ($K_{jj}^*K_{jj}=I_{n_j}$, $j=1\dots, p$).  Moreover, relative
	to the nonzero singular values, $K_{jj}$ must satisfy 
	$K_{jj}^2=\bar H_{jj}$, where $H=V^TQ$, whereas if $\sigma_p=0$, then $K_{pp}$ can be any 
	unitary matrix in $\C^{n_p\times n_p}$.
\end{lem}

It must be noted that in Lemma \ref{UniquenessTakagi} any unitary square root $K_{jj}$ of $\bar H_{jj}$
can be selected to obtain a Takagi factorization.

Finally, below we give a simple algorithm, Algorithm \ref{TakAlgo}, to obtain a Takagi factorization 
when $A$ is nonsingular and with distinct singular values.
The construction below was proposed in \cite[Exercise 15, p.423]{HJ1} and further
used in \cite{CT-Takagi}, to obtain the Takagi factorization of $A$, 
and is based on the SVD of $A$.  An alternative numerical method is given in 
\cite{XuQiao-Takagi}.

\algo{TakAlgo}{Takagi factorization via the SVD} {
	Input:  $A\in \Cnxn$ nonsingular, symmetric: $A=A^T$, and with
	distinct singular values. \\[1ex]
	Output:  $U\in \Cnxn$ unitary, and $S=\diag(\sigma_1,\dots, \sigma_n)$,
	such that $\sigma_1> \sigma_2>  \dots > \sigma_n>0$
	and $A=USU^T$.\\ 
\\	
	1. Let $A=U\Sigma V^*$ be the SVD of $A$ with $U,V\in \Cnxn$ unitary,
	and ordered singular values  in $\Sigma$.  Let $S=\Sigma$ and
	observe that $U^*A\bar U$ is diagonal.\\
	2. Let $\Phi=\diag(\phi_1, \dots, \phi_n)$, where $\phi_k\in \R$, and solve
	the system $\Sigma e^{2i\Phi} =U^*A\bar U$ to obtain 
	$e^{2i \phi_k}=\frac{(U^*A\bar U)_{kk}}{\sigma_k}$, $k=1,\dots, n$.
	Take a square root of $e^{2i \Phi}$ and obtain $e^{i\Phi}$. (Any complex
	square root is a legitimate choice). \\
	3.  Correct the unitary factor $U$: $U\gets Ue^{i\Phi}$.\\
}

\begin{rem}
Note that, in the case of $A$ invertible, and distinct singular values, 
	the factors $S$ and $U$ in a Takagi factorization of $A$ are unique
    up to the signs of the columns of $U$.  
\end{rem}

In this work we are interested in studying the Takagi factorization of a complex and
symmetric matrix valued function {\bf smoothly depending on parameters}.   We will
be concerned with two issues: (i) smoothness of the factors, and (ii) degenerate
(i.e., multiple or zero) singular values.  In the latter case, we will also device (and test) numerical
methods to find the parameter values where degeneracies occur.
A plan of the paper is as follows.
In Section \ref{smoothness} we will discuss the general concern of smoothness with
respect to parameters, and the crucially important concern of the codimension of
having a pair of degenerate singular values, showing that the codimension is 2 both for equal singular values and a zero singular value.
In Section \ref{1paraDEs} we will look at the
1-parameter case, derive differential equations for the factors, and further use
these in the context of a predictor-corrector algorithm to compute a smooth path
of Takagi factorizations.  Finally, in Section \ref{Numerics} we report on numerical
experiments aimed at computing parameter values where matrices smoothly depending on 2 parameters have degenerate singular values.

\section{Smoothness results, degenerate singular values, codimension.}\label{smoothness}
We first give a general result on the co-dimension of having multiple singular values, or
a $0$ singular value.

\begin{thm}\label{Cod2}
Consider $A\in \Cnxn$, symmetric.  Then, having two equal singular values is a (real) co-dimension 2
phenomenon.  Furthermore, having one singular value equal to $0$ is also a co-dimension 2
phenomenon.  
\end{thm}
\begin{proof}
Note that the (real) dimension of the space of
complex symmetric matrices is $2(n(n-1)/2+n)=n^2+n$.  Now, let $A=USU^T$ be a Takagi
factorization of $A$.  

The set $W$ of complex symmetric matrices with a double singular value is the 
image of the map $(U,S) \to USU^T$ , where $U$ is unitary and the diagonal matrix $S$ has
$\sigma_1=\sigma_2$. 
The dimension of the manifold of unitary matrices is $n^2$.  But, if we take 
$\tilde U = UQ$, $Q=\smat{c & s & \\ -s & c & \\ & & I}$, 
where $c^2 + s^2 = 1$ , then $\tilde US \tilde U^T = USU^T$ and
$\tilde U$ is still unitary. Hence the maximum dimension of the strata of
$W$ is $n^2-1+n-1=n^2+n-2$.  Thus, $W$ has co-dimension two.

Next, consider the set $W$ of complex symmetric matrices with a $0$ singular
value.  This is the image of the map $(U,S) \to USU^T$ , where $U$ is unitary and
the diagonal $S$ has $\sigma_n=0$.  If we now take
$\tilde U = UQ$, $Q=\smat{I_{n-1} & \\ &\eta }$, $\abs{\eta}=1$, 
then $\tilde US \tilde U^T = USU^T$ and
$\tilde U$ is unitary. Hence the maximum dimension of the strata of
$W$ is $n^2+n-1-1=n^2+n-2$.  Thus, also in this case $W$ has co-dimension two.
\end{proof}
\begin{rems}$\,$
\begin{itemize}
\item[(i)]
The result of Theorem \ref{Cod2} is surprising, because for general $A\in \Cnxn$, 
having a pair of equal singular values is
a co-dimension 3 phenomenon!  We note that when $A\in \Rnxn$, having two equal singular values
is a co-dimension 2 phenomenon, and hence the impact of symmetry in this complex case is to
lower the co-dimension of having two equal singular values to be like in the real case.
\item[(ii)]
This lowering of co-dimension phenomenon does not occur insofar as loss of rank ($0$ 
singular value) is concerned.  Indeed,  also for a general $A\in \Cnxn$, having a
$0$ singular value is a co-dimension $2$ phenomenon.  However, for $A\in \Rnxn$ having a
$0$ singular value is a co-dimension $1$ phenomenon.  So, the impact of symmetry in this
complex case is not to bring
the co-dimension of having a $0$ singular value to be equal to that of the real case.
\item[(iii)]
It must be appreciated that having a pair of singular values not just equal {\bf but} equal to a
specific value is a higher co-dimension occurrence.  For example, having a pair of $0$
singular values is a phenomenon of (real) co-dimension 6!  
($A\in \C^{2\times 2}$, $A=A^T$, has two $0$ singular values only if $A=0$.)
\end{itemize}
\end{rems}

An implication of Theorem \ref{Cod2} is that, generically, 1-parameter functions will have no
double singular value, and will be full rank.
For this reason, we next show that a smooth full rank $A$ with distinct singular values
has a Takagi factorization with smooth factors.  To prove this result, we will need the
following elementary result stating that a smooth
complex valued function of real variable, of modulus $1$,
has a smooth square root of modulus $1$. 
\begin{lem}\label{SmoothRoot}
	Let $\eta \in \cont^k(\R,\C)$, $k\ge 1$ (or even $\eta \in \cont^\omega$, i.e., analytic), 
	be of modulus $1$ for all $t\in \R$.  Then, there exists a square root
	function $s\in \cont^k(\R,\C)$ of modulus $1$ such that $s^2=\eta$; $s$ is unique up to sign. 
\end{lem}
\begin{proof}
	This is a consequence of the fact that there is a unique smooth logarithm of
	$\eta(t)$, that is a unique smooth real valued function $\psi(t)$ so that $\eta(t)=e^{i\psi(t)}$,
	from which we will have $s(t)=e^{i\psi(t)/2}$.
\end{proof}

\begin{thm}\label{SmoothTakagi}
Let $A\in \cont^k(\R,\Cnxn)$ with $k\ge 0$, $A^T=A$, and 
suppose that for any given $t$ the singular values are distinct and positive.  
Then, $A$ admits a Takagi factorization, $A(t)=U(t)S(t)U^T(t)$, $\forall t$,
with $U$ unitary and $S=\diag(\sigma_1,\dots, \sigma_n), \,\ \sigma_1> \dots > \sigma_n>0$,
where $U,S$ are as smooth as $A$.
\end{thm}
\begin{proof}
	The result is a consequence of the construction given in Algorithm \ref{TakAlgo},
	under the present assumptions.  In fact, see \cite{DieciEirola}, functions with distinct
	singular values have an SVD with factors as smooth as $A$; that is, in the
	decomposition $A(t)=U(t)\Sigma(t)V^*(t)$, $U,\Sigma, V$ are as smooth as $A$. In particular,
	the diagonal function $D=\diag\left(\frac{(U^*A\bar U)_{kk}}{\sigma_k} \,,\ k=1,\dots, n\right)$, is as
	smooth as $A$ and unitary.  Then, using Lemma \ref{SmoothRoot},
	we take a smooth square root of the entries on the diagonal
	of $D$ and obtain the stated smoothness result.
\end{proof}

\begin{rem}
	Under the assumption of Theorem \ref{SmoothTakagi},
	we observe that there is no freedom left in specifying the factors of the
	Takagi factorization, aside from a trivial change in sign of the columns of $U$.
	If we fix a reference Takagi factorization at some value, say at $t=0$,
	then there is only one $\cont^k$ factorization passing through these given conditions.
\end{rem}	

The most important consequence of Theorem \ref{Cod2} is that we need to consider functions
of 2 real parameters
to expect having a pair of equal singular values or a zero singular value.  This is the same situation as what we 
have for coalescing singular values in the real (not complex) case;
e.g., see \cite{DiPu1}.   Furthermore, these occurrences are
expected to take place at isolated parameter values, and to persist under generic perturbations.
For these reasons, below we will focus on matrices depending (smoothly) on two
parameters and we will address the concerns of locating parameter values
where there are double singular values or where a singular value vanishes.
We will call \emph{conical intersections} the locations where two singular values
become equal.

We next see that --similarly to the real symmetric eigenproblem-- the columns of $U$
associated to a pair coalescing eigenvalues will undergo a change in sign (they flip over)
when we cover a loop (i,e., a closed curve) in parameter space, enclosing the
associated conical intersection.  Similarly, we will see that the column of $U$ associated to
a $0$ singular value inside the loop will undergo a change in sign as we cover the loop.
As previously remarked, we should not be concerned
with the possibility of encountering a conical intersection, or a $0$ singular value,
as we cover the 1-d loop, since coalescings and losses of rank are co-dimension 2
phenomena, not 1.
 
To obtain the result on coalescing singular values, and losses of rank, we adopt the following rewriting
(e.g., used in \cite[Theorem 2.7]{ChernDieci}). We write $A\in \Cnxn$, $A=A^T$, as $A=B+iC$, $B,C\in \Rnxn$
symmetric.  Then, we associate to $A$ the following (special) symmetric matrix $M\in \R^{2n\times 2n}$:  
\begin{equation}\label{Mmat}
M\ = \ \bmat{B & C \\ C & -B}\ ,\quad B=B^T\ ,\quad C=C^T\ .
\end{equation}
Now we make the following observations.
\begin{itemize}
\item[(i)]
The eigenvalues of $M$ are given by $\diag(S)$ and $\diag(-S)$ (i.e., the singular values of $A$ and
their opposite values).  We will label the eigenvalues of $M$ as
$\lambda_1\ge \dots \lambda_n \ge \lambda_{n+1} \ge \dots \ge  \lambda_{2n}$, where
$\lambda_j=\sigma_j$, and $\lambda_{n+j}=-\sigma_{n-j+1}$, $j=1,\dots, n$.
\item[(ii)]
If $U$ is a unitary Takagi factor of $A$, writing $U=V+iZ$, then $W=\bmat{V & Z \\ Z & -V}$
gives an orthogonal eigendecomposition of $M$: 
$W^TMW=\bmat{S & 0 \\ 0 & -S}$.  Viceversa, if $M\bmat{X\\ Y}=\bmat{X\\Y} S$, with 
$\bmat{X\\ Y}$ orthonormal, then also $M\bmat{Y\\ -X}=\bmat{Y\\-X}(-S) $, 
$W=\bmat{X & Y \\ Y & -X}$ is an orthogonal eigendecomposition of $M$,
and $U=X+iY$ is unitary and gives a Takagi factorization of $A$. 
\end{itemize}

In light of the correspondence just outlined between the Takagi factors and the
symmetric eigenproblem, see \eqref{Mmat}, we can now exploit known results
about conical intersections of eigenvalues of symmetric functions of two parameters
to obtain the corresponding results for the coalescing singular values of a Takagi
factorization.    We give the main result in the next Theorem, whose proof amounts to
putting together known results for the symmetric eigenproblem, relative to {\sl{generic}}
coalescing points; we refer to \cite{DiPu1} for the meaning of generic in this context\footnote{in essence,
it means that the surfaces of eigenvalues of $M$ come together as a double cone at points
where the eigenvalues are equal}, and we note that to say that the smallest eigenvalue vanishes
at a generic point means that the ordered $n$-th and $(n+1)$-st eigenvalues of $M$ have a
generic coalescing there.

\begin{thm}\label{MainThm}
Let $A\in\cont^k(\Omega,\Cnxn)$ with $k\ge 0$, $A^T=A$, and $\Omega$ a closed and bounded subset of $\R^2$. Let $\sigma_1\ge \dots \ge \sigma_n\ge 0$ be the
continuous singular values of $A$ in $\Omega$.
Let $\Gamma$ be a simple closed curve in $\Omega$, parametrized as a smooth
function $\gamma(t)$, $t\in [0,1]$, and assume that the singular values of $A$ are
nonzero and distinct along $\Gamma$.   Let $A(t)=A(\gamma(t))$, $t\in [0,1]$, be the restriction of
$A$ to $\Gamma$.
Let $A(0)=U_0S_0U^T_0$ be a given Takagi
factorization of $A(0)$, and let $U(t)$, $S(t)$, be the unique smooth Takagi factors
of $A(t)$, for all $t \in [0,1]$, satisfying $U(0)=U_0$, $S(0)=S_0$.  Partition 
$U(t)=\bmat{u_1(t) & \dots & u_n(t)}$.

Finally, assume that there is a unique, generic, coalescing point $x\in \Omega$, where
$\sigma_j=\sigma_{j+1}$, for some $j=1,\dots, n-1$, and that there is no other point in $\Omega$ where a pair of singular values coalesce or a singular value is zero. Then, 
\begin{equation}\label{CI1}
u_k(1)=u_k(0)\ ,\quad \text{for} \quad k\ne  j, j+1\ , \quad \text{and} \quad  
\bmat{u_j(1) & u_{j+1}(1)}=-\bmat{u_j(0) & u_{j+1}(0)}\ .
\end{equation}
Similarly, if $\sigma_n$ vanishes at a unique, generic, point inside $\Omega$, and it is
$\sigma_n\ne \sigma_{n-1}$ in $\Omega$, then  $u_n(1)=-u_n(0)$.
\end{thm}
\begin{proof}
Under the given hypotheses, the function $M$ in $\Omega$ has the continuous and
ordered eigenvalues $\sigma_1\ge \dots \ge \sigma_n \ge -\sigma_n \ge \dots \ge -\sigma_1$,
and along $\Gamma$ these are distinct and $\sigma_1> \dots > \sigma_n>0$.
	
Now, in the case of two equal singular values at a generic point $x\in \Omega$, we will have both
$\lambda_j(x)=\lambda_{j+1}(x)$ and $\lambda_{2n-j+1}(x)=\lambda_{2n-j}(x)$ (since the latter
are the negative of the singular values).  Then, appealing to  Theorem 3.3 in \cite{DiPu1}
for two pairs of coalescing eigenvalues of symmetric
functions, given the form \eqref{Mmat}, and noticing that a coalescing must occur within the
first $n$ ordered eigenvalues of $M$, we obtain the desired result: 
$\bmat{u_j(1) & u_{j+1}(1)}=-\bmat{u_j(0) & u_{j+1}(0)}$.

As far as the vanishing singular value, we again reason with the enlarged system
\eqref{Mmat}, and now notice that a $0$ singular value reflects in the coalescing pair $\lambda_n$
and $\lambda_{n+1}$ as eigenvalues of $M$.  In this case, we can appeal to Theorem 2.8 of \cite{DiPu1}
to conclude that $u_n(1)=-u_n(0)$.
\end{proof}

%
\begin{rem}\label{rem:multiple_degeneracy}
Again by appealing to the enlarged system \eqref{Mmat}, the case of multiple (generic) points of degeneracy of the singular values in $\Omega$ can be analyzed similarly to the way that multiple points of coalescence for eigenvalues were dealt with in \cite{DiPu1}, to which we refer for details. In essence, one can show that each generic degeneracy causes the eigenvector(s) involved in the degeneracy to change sign (see Eq. \eqref{CI1}). For instance, suppose that $\sigma_1=\sigma_2$ and $\sigma_2=\sigma_3$ at two distinct generic coalescing points in $\Omega$, and that there are no other points of degeneracy for the singular values in $\Omega$. Then, the eigenvectors $u_1(t), u_2(t), u_3(t)$ smoothly continued around a 1-periodic loop that contains the two coalescing points will satisfy:
\begin{equation*}
 \bmat{u_1(1) & u_2(1) & u_3(1)}=\bmat{-u_1(0) & u_2(0) & -u_3(0)},
\end{equation*}
because both pairs $u_1, u_2$ and $u_2,u_3$ will change sign. Similarly, if $\sigma_{n-1}=\sigma_{n}$ and $\sigma_n=0$ at two distinct points and $\sigma_1> \dots > \sigma_n>0$ elsewhere, then:
\begin{equation*}
 \bmat{u_{n-1}(1) & u_n(1)}=\bmat{-u_{n-1}(0) & u_n(0)}.
\end{equation*}
\end{rem}

\section{One parameter case.}\label{1paraDEs}

As we saw, under the conditions of Theorem \ref{SmoothTakagi}, a 1-parameter function $A$ admits a smooth Takagi factorization.
Next, we are going to derive differential equations describing the evolution of the factors $U$ and $S$,
and then we will give an algorithm to compute a smooth Takagi factorization
for 1-parameter functions, under the assumption of distinct and nonzero singular values.

\subsection{Differential equations for the factors.}
We have the following result.

\begin{thm}\label{TakagiDE}
Let $A\in \cont^k([0,\infty),\Cnxn)$ with $k\ge 0$, and let $A^T=A$.
Suppose that for any given $t\ge 0$ the singular values of $A(t)$ are distinct and positive.  Let $A(0)=U_0S_0U_0^T$
be a Takagi factorization at $t=0$.
Then, the smooth factors $U$ and $S$ for which $A(t)=U(t)S(t)U^T(t)$, for 
all $t\ge 0$, with $U$ unitary, 
$S=\diag(\sigma_1,\dots, \sigma_n), \,\ \sigma_1> \dots > \sigma_n>0$, $U(0)=U_0$ and $S(0)=S_0$,
satisfy the differential equations 
\begin{equation}\label{edo_pred}\begin{split}
\dot S& = \Real(\diag(U^*\dot A\bar U))\\
\dot U& = UH\ ,
\end{split}\end{equation}
where 
$H=U^*\dot  U$ is the
skew-Hermitian  function such that 
\begin{equation}\label{H_edo}
\begin{split}
H_{kk}&=i \frac{\Img(U^*\dot A\bar U)_{kk}}{2 \sigma_k}\ ,\,\,\quad k=1,\dots, n\ ,\\
H_{kj}&=\frac{\Real(U^*\dot A\bar U)_{kj}}{\sigma_j-\sigma_k}\,+ \, 
i \frac{\Img(U^*\dot A\bar U)_{kj}}{\sigma_j+\sigma_k}\ ,\quad
\quad k=1,\dots, n\ ,\,\ j=1,\dots, k-1\ .
\end{split}
\end{equation} 
\end{thm}
\begin{proof}	
We differentiate the relation $A(t)=U(t)S(t)U^T(t)$ to obtain
\begin{equation*}\begin{split}
\dot A&=\dot U SU^T+U\dot SU^T +US\dot U^T \\
U^*\dot A \bar U& =
U^*\dot US U^T\bar U+U^*U\dot S U^T\bar U+U^*US\dot U^T\bar U.
\end{split} \end{equation*}
Now, using $U^*U=I$ (hence $U^T\bar U=I$), and letting $H=U^*\dot U$ which is
skew-Hermitian and hence $\dot U^T\bar U=-\bar H$, we have
\begin{equation}\label{DE1}
\dot S=U^*	\dot A \bar U-HS+S\bar H\ ,\quad \dot U=UH\ .
\end{equation}
Now we use the fact that $S$ is diagonal, $S=\diag(\sigma_k, \, k=1,\dots, n)$,
to obtain for the singular values
\begin{equation*}\begin{split}
\dot \sigma_k& = (U^*\dot A\bar U)_{kk}-H_{kk}\sigma_k+\sigma_k\bar H_{kk}
\quad\text{or}\quad \\
\dot \sigma_k& = (U^*\dot A\bar U)_{kk}-2H_{kk}\sigma_k\ .
\end{split}\end{equation*}
Now, we must have $H_{kk}=i\phi_k$, where $\phi_k\in \R$, and $\sigma_k\in \R$.
So, we let $(U^*\dot A\bar U)_{kk}=\alpha_{kk}+i\beta_{kk}$ and thus obtain
$\dot \sigma_k =\alpha_{kk}+i(\beta_{kk}-2\phi_k\sigma_k)$ and so
\begin{equation}\label{Sde}
\dot \sigma_k = \alpha_{kk}\ ,\quad \phi_k=\frac{\beta_{kk}}{2\sigma_k}\ ,
\quad k=1,\dots, n\ .
\end{equation}
At the same time we must have $(\dot S)_{kj}=0$, $k\ne j$, and we will use this to
obtain the remaining relations for the entries of $H$.  Namely, we have
$$(U^*\dot A\bar U)_{kj}-H_{kj}\sigma_j+\sigma_k \bar H_{kj}=0$$
and letting $H_{kj}=\alpha_{kj}+i\beta_{kj}$ we have
$$(U^*\dot A\bar U)_{kj}=\alpha_{kj}(\sigma_j-\sigma_k)+
i\beta_{kj} (\sigma_j+\sigma_k)$$
and therefore
\begin{equation}\label{Ude}
\alpha_{kj}=\frac{\Real(U^*\dot A\bar U)_{kj}}{\sigma_j-\sigma_k}\ ,\,\,
\beta_{kj}=\frac{\Img(U^*\dot A\bar U)_{kj}}{\sigma_j+\sigma_k}\ ,\quad
\quad k=1,\dots, n\ ,\,\ j=1,\dots, k-1\ .
\end{equation}
Using \eqref{Sde} and \eqref{Ude} in \eqref{DE1}, we obtain the differential
equations governing the evolution of the factors $S$ and $U$ in the
Takagi factorization of $A$.
\end{proof}

\begin{rem}
The assumption of nonzero singular values has been used in \eqref{Sde} to obtain
uniquely the phases $\phi_k$'s.
The assumption of distinct singular values has been used
in \eqref{Ude} to obtain the values $H_{kj}$.
\end{rem}

\subsection{A predictor corrector 
	algorithm to compute the factorization along loops in parameter space.}\label{PredCorrAlgo}

Following the lines of our previous work on eigenvalue problems, see \cite{DPP2}, we developed
a predictor-corrector procedure to compute the smooth Takagi factorization of one
parameter symmetric functions.  

Given the factorization  $A(t)= U(t) S(t) U^T(t)$ at some $t$,
 we briefly describe how we compute
the smooth factors $U$ and $S$ at a new point $t_{new}=t+h$. 
The diagonal factor $S(t_{new})$ 
and a unitary matrix $U_{new}$ such that $A(t_{new}) = U_{new}\, S(t_{new})\, U^T_{new}$ 
are obtained by Algorithm \ref{TakAlgo}, using  standard software like the {\tt svd Matlab} command. 
Recalling    that $U_{new}$ is unique up to the signs of its columns, correct signs for
  $U(t_{new})$ can be
predicted using the differential equations.
Indeed, taking an Euler step  in 
\eqref{edo_pred}-\eqref{H_edo}, we first have
$$ U(t+h)\simeq U(t) (I+h H(t)) .$$
 Since  $H(t)$  depends on the matrix
$U^*(t)\dot A(t)\bar U(t)$, we replace
  the  derivative ${\dot A}(t)$  with $(A(t+h)-A(t))/h$,   
$$U^*(t)\dot A(t)\bar U(t)\simeq 
\frac{U^*(t) A(t+h)\bar U(t)-S(t)}{h}.$$
Setting  $A_U= U^*(t) A(t+h) \bar U(t)$   we finally have
\begin{equation}\label{predittore_U}
	\begin{split}
		U^{pred}& = U(t)(I+\mathcal H), 
		\quad  {\rm with}\; \mathcal H\;\; \text{skew-Hermitian~and~such~that}\\
		\mathcal H_{kk}&=i \frac{\Img(A_U)_{kk}}{2 \sigma_k}\ ,\,\,\quad k=1,\dots, n\ ,\\
		\mathcal H_{kj}&=\frac{\Real(A_ U)_{kj}}{\sigma_j-\sigma_k}\,+ \, 
		i \frac{\Img(A_U)_{kj}}{\sigma_j+\sigma_k}\ ,\quad
		\quad k=1,\dots, n\ ,\,\ j=1,\dots, k-1\ .
	\end{split}
\end{equation}
Then, 
we compute  the
sign matrix
$Z={\rm sign}(\Real(\diag(U_{new}^* \, U^{pred}))$ 
to correct the signs of the singular vectors, and set  $U(t_{new})=U_{new}\, Z$. 
Clearly, the correction is reliable as long as the stepsize $h$ is sufficiently small with respect 
to factors variation, which we monitor by computing the   
following parameters:
\begin{equation}\label{rholv} 
	\rho_U   =\frac{ \|U(t_{new})-U^{pred}\|_F}{\sqrt{n}}
	\quad {\rm and}  \quad
	\rho_{\sigma}   = \max_j   \frac{|\sigma_{j}(t_{new})-\sigma_{j}^{pred}|} {|\sigma_{j}(t_{new})|+1} \;;
\end{equation}
predicted singular values $\sigma_j^{pred}$, $j=1,\ldots,n$,   
are obtained proceeding  as before
for $U^{pred}$:
\begin{equation}\label{predittore_S}
	S(t+h)\simeq S(t)+h\,  \Real(\diag(U^*(t)\dot A(t)\bar U(t)))
	\simeq \Real(\diag(A_U)) \equiv S^{pred} .
\end{equation}
We show in Algorithm \ref{sfs}  how $\rho_U$ and $\rho_{\sigma}$
are used    to update the stepsize $h$ (see Step 4),  and to   accept   a step  or  declare {\it failure}  (see Step  5). 

Finally,  before attempting a new step after  an accepted one, 
we use the following  simple secant approximation of the singular values at $t_{new}+h$: 
\begin{equation*} \label{secant}
	\sigma_j^{sec}=\sigma_j(t_{new})+h\, \dot{\sigma}_j^{sec}= \sigma_j(t_{new})+
	h\; \frac{ \sigma_j(t_{new})-\sigma_j(t)} { t_{new}- t} \simeq   \sigma_j(t_{new}+h) , \;\;\; j=1,\ldots,n,
\end{equation*}
in order to reduce the occurrence of failures. Recall that we expect to
have $\sigma_1>\sigma_{2}>\ldots > \sigma_n>0$: then   
the new step $t_{new}+h$ is likely to fail if either $\sigma^{sec}_n<0$ or $\sigma^{sec}_j<\sigma^{sec}_{j+1}$ for some $j$.
In these cases $h$ is   reduced 
- at least halved - as 
shown at Steps 7 and 8 of Algorithm \ref{sfs}. 

\algo{sfs}{Predictor-Corrector step } 
{Input:  $t$,  the Takagy factorization  $A(t)=U(t)S(t) U^T(t)$, a  stepsize $h$.\\[1ex]
	Output:  $t_{new}$,  the smooth factors  $U(t_{new})$ and $S(t_{new})$, 
	an updated stepsize $h$.\\[1ex]
	1. Set $t_{new}=t+h$,  and compute ~~$U^{pred}$ ~and~ 
	$S^{pred}$ by \eqref{predittore_U} and \eqref{predittore_S};\\
	2. Compute a Takagi factorization
	$A(t_{new})=U_{new}\,S(t_{new}) \, U^T_{new}$ by Algorithm \ref{TakAlgo}; \\
	3.  Set  $U(t_{new}) = U_{new}\, Z$, ~ with~ $Z={\rm sign}(\Real(\diag(U_{new}^* \, U^{pred}))$;\\
	4.   Compute   $\rho_{\sigma}$ and $\rho_U$ from \eqref{rholv}, 
	set $\rho=  \max \{\rho_{\sigma}, \rho_U \}/{\tt tolstep}$ (in our experiments, we
	have used $\mathtt{tolstep}=  10^{-2}$), 
	and update $h=h/\rho;$ 
	\vskip 2pt
	5. If $\rho\le 1.5$, accept the step; \\
	\hspace*{10pt} otherwise declare {\it failure}, go to step 1,  and retry  with the new (smaller) $h$;\\
	6.   Compute   $\displaystyle  \dot{\sigma}_j^{sec} =   \frac{ \sigma_j(t_{new})-\sigma_j(t)} { t_{new}- t}\;$
	~ and ~
	$\;\sigma_j^{sec} = \sigma_j(t_{new})+h   \dot{\sigma}_j^{sec},$ ~ for $j=1,\ldots,n;$    \\
	7. If   ~$\sigma^{sec}_j<\sigma^{sec}_{j+1}$ for some $j=1,\ldots,n-1$, set $h =   \min \left \{ \,  
	h/2,\, 0.9\; h_{sec}  
	\right \}$, where
	\begin{equation*}
		h_{sec}= \min \left \{ 
		\,  
		\frac{ \sigma_j(t_{new})-\sigma_{j+1}(t_{new})} { \dot{\sigma}^{sec}_{j+1}-\dot{\sigma}_j^{sec}}, ~
		{\rm for} ~\sigma^{sec}_j<\sigma_{j+1}^{sec}, ~j=1,\dots, n-1 
		\right \};
	\end{equation*}
         8. If ~$\sigma_n(t_{new})+h   \dot{\sigma}_n^{sec} <0$,~~ 
	set ~ ~$h  =  \min \left \{ \, \displaystyle
	h/2,    \; 0.9\; \sigma_n(t_{new})\,/\,|\dot{\sigma}^{sec}_{n}|
	\right \}$ .
}

Following a smooth path 
can typically force the algorithm to take small steps, 
possibly very small steps,   when it passes near a point  where either $\sigma_n$ vanishes  or
two singular values coalesce, depending on the distance from such degeneracies. 
This behaviour
is easily justified looking  again at the differential equations \eqref{edo_pred}-\eqref{H_edo}.
Indeed, since $H_{nn}$ and the real part of the matrix $H$ are  inversely
proportional, respectively, to $\sigma_n$ and to the difference between two singular values,
when one of these quantities is small  some rapid variation in the unitary factor $U$ has to be expected.
A lower bound $h_{min}$ for the stepsize is then recommended, 
below which the continuation procedure has to be halted. In our experiments we set $h_{min}=100\,\mathtt{eps}$.

\section{Numerical tests.}\label{Numerics}
We should appreciate that the sole 
knowledge that the co-dimension of having degenerate singular values is 2 does not give any insight
into how many coalescing points we should expect to have in a spatial region of parameter
space.  Yet, the latter is exactly the kind of knowledge that one wants: how many isolated
coalescings should we expect?  For this reason, 
the goal of this section is to provide some insight into the so-called {\emph{density of degeneracies}}
for the Takagi factorization, relative to an appropriate ensemble of random matrices depending
on two parameters.

In our experiments, we built several realizations of 2-parameter matrix valued functions of increasing size, and performed a statistical study on the number of coalescings and losses of rank found in a certain domain. Our matrix valued functions are defined as trigonometric combinations of four matrices $A_1, \ldots, A_4$:
\begin{equation}\label{def:RandMatFun}
A(x,y)=A_1\cos(x)+A_2\sin(x)+A_3\cos(y)+A_4\sin(y),\ (x,y)\in \R^2,
\end{equation}
where each matrix $A_k$ is generated as follows:
\begin{enumerate}
\item  take a random matrix $B\in\C^{n\times n}$ with entries $B_{ij}=x_{ij}+\mathrm{i}\,y_{ij}$, where $x_{ij}$ and $y_{ij}$ are i.i.d. selected from $N_\R(0,\frac{1}{4})$,
\item form $A_k=B+B^T$.
\end{enumerate}

Matrix functions like \eqref{def:RandMatFun} had been considered in \cite{WA} in a study 
of degeneracies for real symmetric random matrices.  
Note that a function $A(x,y)$ like in \eqref{def:RandMatFun} is
defined (and analytic) everywhere in $\R^2$. Moreover, being each matrix $A_k$ 
symmetric with off-diagonal entries in $N_\C(0,1)$ and diagonal entries in $N_\C(0,2)$,
the same is true for $A(x,y)$ for all values of $x$ and $y$, except that the variances 
are two times those of the corresponding entries of the matrices $A_k$'s.

Hereafter, we will denote by $\mathcal{G}$ the ensemble of random symmetric matrices 
of interest, defined as follows:
\begin{equation*}
\mathcal{G} = \left\{A\in \Cnxn :\ A=A^T=(A_{ij})_{i,j=1}^n \,\text{such that}\,\ 
\left\{
\begin{array}{ll}
a_{ij}\in N_\C(0,2) & \text{if }i = j\\
a_{ij}\in N_\C(0,1) & \text{otherwise}
\end{array}
\right.
\right\}.
\end{equation*}
\begin{thm}\label{thm:Invariance}
Let $U\in\C^{n\times n}$ be unitary. Then, if $A\in\mathcal{G}$, also $U^TAU\in\mathcal{G}$.
 \end{thm}

\begin{proof} Let $A\in\mathcal{G}$. It is straightforward to see that all entries of $U^TAU$ are normally distributed with mean zero. Let us compute their variances.
\smallbreak
\noindent \underline{Diagonal entries}: Let $1\le i\le n$. Then we have:
\begin{equation*}
 \left( U^TAU \right)_{ii} =\sum_{h,k=1}^n A_{hk}U_{hi}U_{ki}=\sum_{h=1}^n A_{hh}U_{hi}^2+
 2 \sum_{\substack{h,k=1 \\ h<k}}^n A_{hk}U_{hi}U_{ki}\,.
\end{equation*}
Given that all entries along and above the diagonal of $A$ are independent, we have:
\begin{equation*}
\begin{split}
\var\left(\left(U^TAU\right)_{ii}\right) & =\sum_{h=1}^n \var\left(A_{hh}\right)\lvert U_{hi}\lvert^4 \,+\,
4 \sum_{\substack{h,k=1 \\ h<k}}^n \var\left(A_{hk}\right) \lvert U_{hi}\vert^2\lvert U_{ki}\lvert^2 =\\
 & = 2\left(\sum_{h=1}^n \lvert U_{hi}\lvert^4 \,+\,
2 \sum_{\substack{h,k=1 \\ h<k}}^n \lvert U_{hi}\vert^2\lvert U_{ki}\lvert^2 \right) = 2\, \lVert U_{:,i} \left(U_{:,i}\right)^*\lVert^2_F = 2\,,
\end{split}
\end{equation*}
where $U_{:,i}$ denotes the $i$-th column of $U$.
\smallbreak
\noindent \underline{Off-diagonal entries}: Let $i,j=1,\ldots,n$, $i\ne j$. We have:
\begin{equation*}
 \left(U^TAU\right)_{ij}  =\sum_{h,k=1}^n A_{hk}U_{hi}U_{kj}=\sum_{h=1}^n A_{hh}U_{hi}U_{hj}+
\sum_{\substack{h,k=1 \\ h<k}}^n A_{hk}\left(U_{hi}U_{kj}+U_{ki}U_{hj}\right)\,.
 \end{equation*}
 Just as before, we can compute the variance of $\left(U^TAU\right)_{ij}$:
\begin{equation*}
\begin{split}
\var\left(\left(U^TAU\right)_{ij}\right) & =\sum_{h=1}^n \var\left(A_{hh}\right)\lvert U_{hi}\lvert^2\lvert U_{hj}\lvert^2 \,+\,
\sum_{\substack{h,k=1 \\ h<k}}^n \var\left(A_{hk}\right) \lvert U_{hi}U_{kj}+U_{ki}U_{hj} \lvert^2 =\\
 & = 2\sum_{h=1}^n \lvert U_{hi}\lvert^2\lvert U_{hj}\lvert^2 \,+\,
\frac{1}{2} \sum_{\substack{h,k=1 \\ h\ne k}}^n \lvert U_{hi}U_{kj}+U_{ki}U_{hj} \lvert^2 =\\
& = \frac{1}{2}\, \lVert U_{:,i} \left(U_{:,i}\right)^T+
U_{:,j} \left(U_{:,j}\right)^T\lVert^2_F = 1\,,
\end{split}
\end{equation*}
\end{proof}
\begin{rem} In a sense,
the ensemble $\mathcal{G}$ is a complex symmetric analog of two classes of random matrices that have been extensively studied: the Gaussian orthogonal and unitary ensembles (GOE/GUE), see \cite{Mehta}. 
It is a well known fact that matrices in the GOE and GUE ensembles are invariant under the similarities $A\mapsto U^TAU$ with $U$ orthogonal and $A\mapsto U^*AU$ with $U$ unitary, respectively. Computations similar to those performed in the proof of Theorem \ref{thm:Invariance} yield the following identities for given $A\in\mathcal{G}$ and unitary $U\in\Cnxn$:
\begin{equation*}
\begin{split}
 \var\left(\left(U^*AU\right)_{ii}\right) & = 2\left( 1- \lVert \Img\left(  U_{:,i} \left(U_{:,i}\right)^* \right) \lVert^2_F \right),\ \text{ for all } i=1,\ldots,n\,. \\
 \var\left(\left(U^*AU\right)_{ij}\right) & = 1,\text{ for all } i, j=1,\ldots,n,\,\text{ with } i\ne j.
\end{split}
\end{equation*}
This shows that $\mathcal{G}$ is not invariant under the unitary similarity $A\mapsto U^*AU$ unless $U$ is real, in which case we have $U^*=U^T$ and fall back into the case of Theorem \ref{thm:Invariance}. 
\end{rem}

We now report on the results of our experiments. We carried out two different sets of experiments, aimed at counting the number of points within a given region $\Omega$ where, respectively, a pair of singular values coalesces and a loss of rank occurs. To study the coalescence of singular values, we built 10 realizations of random matrix functions defined as in \eqref{def:RandMatFun} for each of the following dimensions: $50, 60, \ldots, 120$. For the losses of rank, we built 20 realizations for each of the dimensions $50, 100, \ldots, 500$. 
Because of periodicity of $A(x,y)$, we took the domain
$\Omega=[0, 2\pi]\times [0, \pi]$.  We divided $\Omega$ into a uniform grid made of $1024\times 512$ square boxes and, for each box, we established the presence of a coalescing point of singular values by numerically verifying the change of sign expressed  through equation \eqref{CI1} in Theorem \ref{MainThm}. 
In practice, the verification is based upon the computation of a smooth Takagi factorization around the sides of each box through the algorithm described in Section \ref{PredCorrAlgo}. For the losses of rank, we actually chose a coarser grid of $128\times 64$ boxes, since in this case we observed a much smaller number of events and
had to increase the dimension of the problems in order to obtain statistically relevant data. 
To establish the occurrence of a loss of rank, we relied on the same tools described above but counted only the coalescing points for the pair of eigenvalues $(\lambda_n, \lambda_{n+1})$ of the $2n\times 2n$ matrix function $M(x,y)$ defined in \eqref{Mmat}. 
In practice, however, we did not work with the double--size matrix $M$, but made use of 
the fact that coalescence for this pair of eigenvalues of $M$ is betrayed by an odd 
number of columns of $U$ changing sign upon continuation around the boundary
of a box, see Remark \ref{rem:multiple_degeneracy}.
All computations have been performed on PACE, the HPC infrastructure at the Georgia Institute of Technology, and required a total of about 9K CPU-hours.

The results of the experiments are illustrated in Figure \ref{fig:coal_bestfit}. 
They suggest that the number of coalescing points of singular values grows
{\bf quadratically} in the size of the problem, and that the number of points
where there is a loss of rank grows {\bf linearly}.

\begin{figure}[ht!]
	\begin{center}
		\includegraphics[width=.99\textwidth]{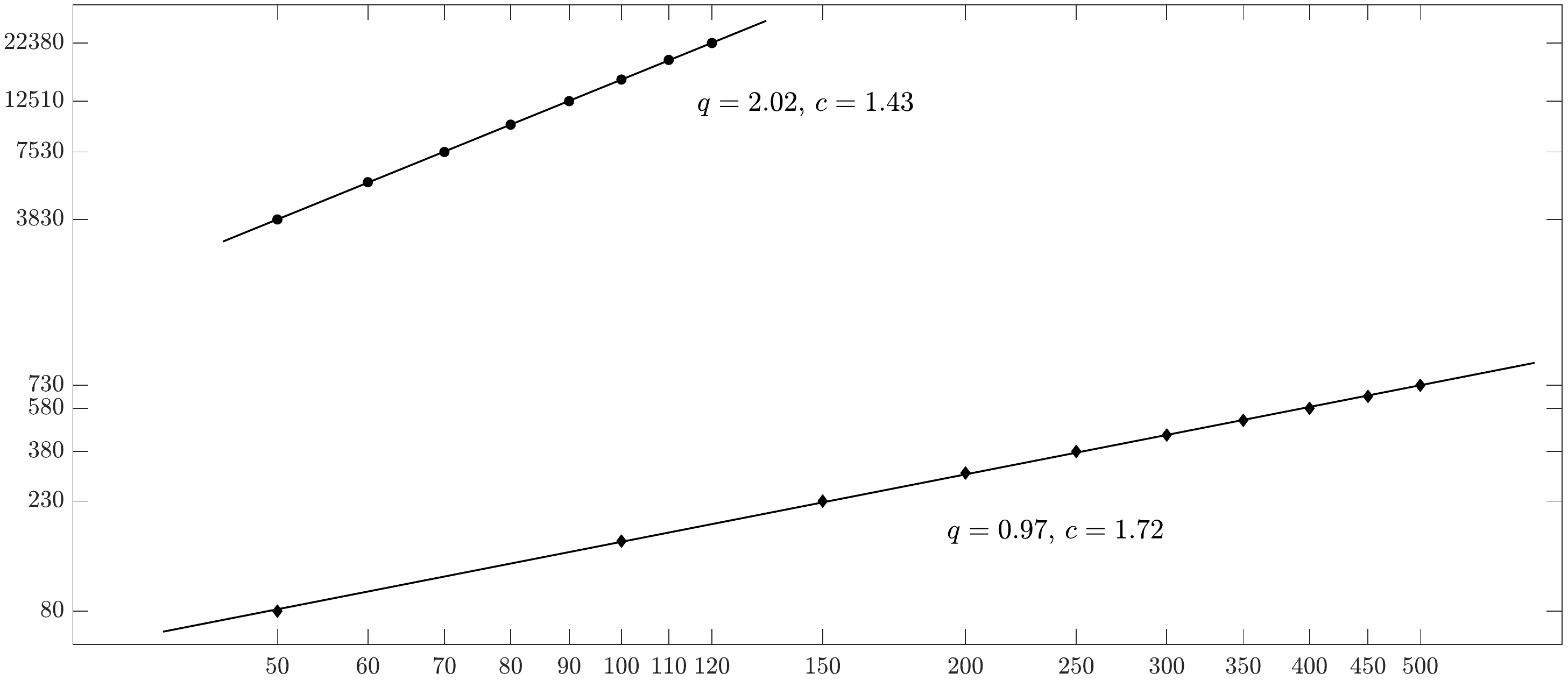}
		\caption{The figure shows the results of a ``log--log least squares'' performed on the data obtained from our experiments (circles for number of coalescences of singular values, diamonds for losses of rank). It also shows the parameters $c$ and $q$ obtained by best fit for the power law: $$\text{ \# of points } = c\,n^q, \text{ where } n \text{ is the matrix dimension.}$$}
		\label{fig:coal_bestfit}
	\end{center}
\end{figure}

\begin{rem} In the physics literature (mostly quantum mechanics), 
for parameter dependent random matrices of a certain ensamble, 
the ``density of degeneracies'' is a probability distribution that provides the 
expected number of points of coalescence per unit of measure in parameter space. 
From the density of degeneracies, one can compute the total number of points of degeneracy expected to be found within a given domain in terms of the size of the matrix functions. 
For the GOE and GUE ensembles, the density of degenerate eigenvalues has been computed in \cite{WA} and \cite{WSW} respectively, and found to be proportional to the square (respectively, cube) of the density of eigenvalues.
\end{rem}

We recall that, in the present work, the term ``degeneracy'' may refer either to coalescence of singular values or to loss of rank. Under fair working assumptions, and through an empirical argument similar to one used in \cite{DPP1}, we make the following:

\begin{claim}\label{claim:powerlaw}
Consider a $n\times n$ random matrix depending on smoothly depending on two or more parameters. Assume that,
for all values of the parameters, its singular values are asymptotically distributed according to the {\bf quarter-circle distribution}
\begin{equation}\label{eq:quartercirc}
\rho(\sigma)=\frac{\sqrt{2}}{\pi}\sqrt{n-\sigma^2/8},\,\quad 
\sigma\in\left[0, \sqrt{8n}\right]\ .
\end{equation}

Assume that the density of coalescing singular values is given by $c\,\rho(\sigma)^p$, where $c>0$ is a constant.  Then,
the expected number of points of coalescence for the singular values per unit of measure in parameter space is asymptotic to $n^{\frac{p}{2}+1}$, while the expected number of points of coalescence involving a particular pair of consecutive singular values is asymptotic to $n^\frac{p}{2}$. Moreover, the number of points of loss of rank is also asymptotic to $n^\frac{p}{2}$.
 \end{claim}

\begin{proof}[Evidence]
Below we present an empirical argument in support of Claim \ref{claim:powerlaw}.
From \eqref{eq:quartercirc}, for the density of singular values $\rho$ we have:
\begin{equation*}
\int_0^{\sqrt{8n}} \rho(\sigma)\, d\sigma = n,
\end{equation*}
for all $n>0$. For any $n>2$ and any $0\le k\le n$, let us consider $s_k$ such that
\begin{equation*}
\int_{s_k}^{\sqrt{8n}} \rho(\sigma)\, d\sigma = k.
\end{equation*}
Then we expect each interval $[s_{k},s_{k-1}]$ to contain exactly 1 singular value of the matrix function, and also expect coalescence for a pair of consecutive singular values to occur at the boundaries of each interval. That is, we expect $\sigma_k=\sigma_{k+1}=s_k$ for $k=1,2,\ldots,n-1$. Therefore, given our assumption on the expression of the density of degeneracies, the expected number of coalescing points involving the pair $(\sigma_k,\sigma_{k+1})$ is $c\,\rho(s_k)^p=c\left(\frac{\sqrt{2}}{\pi}\right)^p\left(n-s_k^2/8\right)^\frac{p}{2}$, which is asymptotic to $n^{\frac{p}{2}}$, and the expected number of coalescing points involving any pair of singular values is:
\begin{equation*}
 c\,\sum_{k=1}^{n-1} \rho(s_k)^p=c\left(\frac{\sqrt{2}}{\pi}\right)^p\, \sum_{k=1}^{n-1}\left(n-s_k^2/8\right)^\frac{p}{2},
\end{equation*}
which is asymptotic to $n^{\frac{p}{2}+1}$. The reasoning on the asymptotic for the number of rank losses follows the same line, except one needs to consider the eigenvalues of the double-size problem \eqref{Mmat}.
\end{proof}

Unfortunately, we could not find any reference for the density of singular values of
random complex symmetric matrices in $\mathcal{G}$, let alone the relation between
density of degenerate singular values and the density of singular values in such case. 
Yet, extensive numerical experiments strongly suggest that the singular values of 
matrices in $\mathcal{G}$ are asymptotically distributed according to the quarter-circular distribution \eqref{eq:quartercirc}. 
In fact, this is known to be the case for general random Gaussian matrices, and was established by Marchenko and Pastur in \cite{MP}. In the GOE/GUE cases, it is a consequence of the well known Wigner's semicircular law. Recently, it has been proved also for the so-called \emph{checkerboard} matrices, see \cite{Miller}.

By virtue of Claim \ref{claim:powerlaw}, the results of our experiments can be used
in support of the following:

\begin{claim}
For the random matrix functions $A(x,y)$ defined in \eqref{def:RandMatFun}, the density of coalescings for the singular values is proportional to $\rho(\sigma)^2$, where $\rho$ is the density of singular values defined in \eqref{eq:quartercirc}. 
\end{claim}

\section{Conclusions}

In this work, we considered the Takagi factorization of a matrix valued function depending on
parameters.  Our main interest was in being able to locate degenerate (i.e., double or zero) 
singular values in  parameter space.  We proved that having a generic coalescing of singular values is
a co-dimension 2 phenomenon, in contrast to the case of equal singular values of a general
complex matrix for which the codimension is 3. At the same time, losses of rank remain codimension 2 phenomena. We further derived a system of differential equations satisfied by a smooth
factorization along 1-d paths, and proposed an algorithm of predictor-corrector type that
exploits these differential equations.  Finally, we implemented a method to locate the
occurrence of degeneracies and made a numerical study of the density of such degeneracies,
discovering that the power law governing the number of coalescing singular values
is of the type $c_1n^2$ whereas the one governing the number of losses of rank is
of the type $c_2n$, for appropriate constants $c_1, c_2$.

\section*{Acknowledgements}
This research was supported in part through research cyberinfrastructure resources and services provided by the Partnership for an Advanced Computing Environment (PACE) at the Georgia Institute of Technology, Atlanta, Georgia, USA. The patronage of INdAM-GNCS is also acknowledged.

\end{document}